\documentclass{amsart}

\usepackage{newtxtext}

\usepackage[initials]{amsrefs}
\renewcommand{\MR}[1]{}
\usepackage{verbatim}
\usepackage{mathtools}
\usepackage{mathrsfs}
\usepackage[hidelinks]{hyperref}
\usepackage{tikz-cd}
\usetikzlibrary{decorations.pathmorphing}
\usepackage{amssymb}
\usepackage{xspace}
\usepackage{amsthm}

\theoremstyle{plain}
\newtheorem{theorem}{Theorem}[section]
\newtheorem{lemma}[theorem]{Lemma}
\newtheorem{corollary}[theorem]{Corollary}
\newtheorem{proposition}[theorem]{Proposition}
\newtheorem{conjecture}[theorem]{Conjecture}
\theoremstyle{definition}
\newtheorem{definition}[theorem]{Definition}
\theoremstyle{remark}

\newcommand{\shapes}{\mathcal{S}}
\newcommand{\orientedShapes}{\mathcal{S}_+}
\newcommand{\compactR}{\hat\R{}}
\newcommand{\pl}{\mathsf{PL}}
\newcommand{\Homeo}{\operatorname{Homeo}}
\newcommand{\vol}{\operatorname{vol}}
\newcommand{\sign}{\operatorname{sign}}
\newcommand{\diam}{\operatorname{diam}}

\newcommand{\Z}{\mathbb{Z}}
\newcommand{\R}{\mathbb{R}}
\newcommand{\C}{\mathbb{C}}
\newcommand{\composed}{{\mathchoice{\,{\scriptstyle\circ}\,}{{\scriptstyle\circ}}
    {{\scriptscriptstyle\circ}}{{\scriptscriptstyle\circ}}}}
\DeclareMathOperator{\Gr}{Gr}
\renewcommand{\d}{\ensuremath{\hspace{2pt} d}}

\title{Indecomposable Quasiconformal Maps of Manifolds}
\author{Benjamin B. McMillan}
\thanks{This work was supported by the Institute for Basic Science (IBS-R032-D1).}
\address{Center For Complex Geometry, Institute for Basic Science, 55 Expo-ro, Yuseong-gu 34126 Daejeon, South Korea}
\email{mcmillan@ibs.re.kr}

\begin{document}
\begin{abstract}
  We demonstrate the existence of quasiconformal mappings on closed manifolds (dimension at least \( 5 \)) that cannot be decomposed as a composition of mappings with arbitrarily small conformal distortion. This answers in the negative the global version of a classical question of Gehring.
\end{abstract}

\maketitle

\section{Introduction}

The aim of this paper is to show two related results, one positive and one negative, on quasiconformal mappings between Riemannian manifolds.
For both statements, fix two closed Riemannian manifolds \( M, N \) and consider quasiconformal homeomorphisms \( f \colon M \to N \).

The positive statement (Theorem \ref{thm: pl approximation of nearly conformal}) is that there exists a critical quasiconformal distortion \( K_{\pl} \) below which any \( M \to N \) can be arbitrarily well approximated by piecewise linear homeomorphisms with respect to the \( C^{0} \) topology.
In fact, the distortion \( K_{\pl} \) can be chosen to depend only on the dimension of \( M \).

The negative statement (Corollary \ref{thm-cor non factorable qc homeos}) uses the positive statement, in conjunction with classical results from differential topology, to demonstrate the existence of quasiconformal mappings \( f \colon M \to M \) that cannot be decomposed as a composition of functions with conformal distortion all below \( K_{\pl} \).
That is to say, in every factorization of \( f \) as a composition of quasiconformal mappings
\[ M \stackrel{f_{1}}{\longrightarrow} M \stackrel{f_{2}}{\longrightarrow} \ldots \stackrel{f_{r}}{\longrightarrow} M , \]
at least one of the \( f_{i} \) must have distortion more than \( K_{\pl} \).
It is worth noting that, due to the positive statement, this remains true even if one allows distinct Riemannian metrics on \( M \) at each step of the factorization.

The key differential topology fact used here (Theorem \ref{thm: non pl approximable homeo}) is the existence of manifolds \( M \) and quasiconformal mappings on \( M \) that cannot be approximated by a \( \pl \)-homeomorphism.
The manifolds \( M \) in question are not exotic, with domains such as \( S^{3} \times S^{n-3} \) sufficient to observe the phenomenon in each dimension \( n \ge 5 \).

Discussed more below, this paper answers the global version of a classical question in quasiconformal geometry. The local question, where \( M \) is an open domain in Euclidean space, remains open.

\subsection{History}

Given a domain \( U \subset \C \) and a mapping \( f \colon U \to \C \), one may measure the conformal distortion of \( f \) by its \emph{dilatation},
\[ D_{f} = \frac{|f_{z}| + |f_{\bar{z}}|}{|f_{z}| - |f_{\bar{z}}|} , \]
which is always at least \( 1 \).
For \( K \ge 1 \), one says that \( f \) is \emph{\( K \)-quasiconformal} if \( D_{f} \) is everywhere at most \( K \).
Quasiconformal maps are closed under composition.
In fact, letting \( K(f) \) denote the minimal bound on dilatation of \( f \), it holds for composable \( f,g \) that dilatation is submultiplicative,
\begin{equation}\label{eq: submultiplicativity of dilatation}
  K(f \composed g) \le K(f)K(g) .
\end{equation}

Maps of small bound on dilatation are close to being conformal, so it is an important feature of the two dimensional theory that any quasiconformal mapping can be decomposed as a composition of maps with arbitrarily small dilatation.
To establish terminology, say that a \emph{\( K \)-small factorization} of \( f \) is a decomposition
\[ f = f_{r} \composed \cdots \composed f_{1} \]
in which each \( f_{1} , \ldots , f_{r} \) is a \( K \)-quasiconformal homeomorphism between Riemannian manifolds.
Say that \( f \) is \emph{tame} if it admits a \( K \)-small factorization for all \( K > 1 \).
One may also ask for the stronger condition, \emph{a minimal \( K \)-small factorization}, for which it holds
\[ K(f) = K(f_{r}) \cdots K(f_{1}) . \]
The measurable Riemann mapping in \( 2 \) dimensions shows that all quasiconformal maps are tame, and in fact admit minimal factorizations.
A good reference for the \( 2 \)-dimensional theory is \cite{Ahlfors1966--LecturesQuasiconformalMappings}.


In the theory of higher dimensional quasiconformal mappings \cites{Gehring-Martin-Palka2017--IntroductionTheoryHigherDimensionalQuasiconformalMappings, Gehring2005--QuasiconformalMappingsEuclideanSpaces, Martin2014--TheoryQuasiconformalMappingsHigherDimensionsI, Vaisala1971--LecturesDimensionalQuasiconformalMappings} there are several distinct but qualitatively equivalent measures of dilatation.
The dilatations generally satisfy the submultiplicativity of Equation \ref{eq: submultiplicativity of dilatation}, so one can again ask whether quasiconformal maps admit small factorizations.
This has remained an important open question, going back to at least Gehring \cite{Gehring1987--TopicsQuasiconformalMappings}.

There exists some progress, militating in both directions.
In \cite{Fletcher-Markovic2012--Decomposing}, Fletcher and Markovic demonstrate that every continuously differentiable mapping of the sphere can be decomposed as a composition of maps of arbitrarily small isometric distortion.
Since small isometric distortion implies small conformal distortion, this demonstrates differentiable sphere mappings as tame.
In contrast, in \cite{He-Liu2019--FactoringHigherDimensionalQuasiconformalMappings}, He and Liu construct maps, starting in dimension \( 3 \), which admit no minimal factorizations for any of the linear, inner, or outer dilatation.
It is worth noting that their examples do admit small factorizations, just not minimal ones.

\subsection{Approximation by piecewise linear homeomorphisms}


Our examples of indecomposable mappings are based on the following Theorem, which relates nearly conformal mappings to piecewise linear mappings (\( \pl \) for short).

\begin{theorem}\label{thm: pl approximation of nearly conformal}
  In each dimension \( d \) there is a conformal dilatation \( K_{\pl} > 1 \) such that every \( K_{\pl} \)-quasiconformal mapping
  \[ f \colon M \to N \]
  between smooth closed Riemannian \( d \)-manifolds can be arbitrarily well approximated by \( \pl \)-homeomorphisms in the \( C^{0} \) function topology.

  Furthermore, for each \( K' > 1 \), there exists \( K > 1 \) such that the \( \pl \)-approximation of any \( K \)-quasiconformal mapping can be chosen to be \( K' \)-quasiconformal.
\end{theorem}

The proof of Theorem \ref{thm: pl approximation of nearly conformal} is given in Section \ref{sec: pl approximation of nearly conformal}.
We recall more details of \( \pl \)-manifolds in Section \ref{sec: background pl mappings}.
Suffice it to say here that it is Whitehead's Theorem \cite{Whitehead1940--C1Complexes} that a smooth manifold structure on \( M \) determines an underlying compatible \( \pl \)-structure, and a map \( f \colon M \to N \) between smooth manifolds is \( \pl \) if it pushes forward the compatible \( \pl \)-structure of \( M \) to that of \( N \).

To say that \( f \) is \( C^{0} \) approximated by \( \pl \)-homeomorphisms means that for each \( \epsilon > 0 \) there exists a \( \pl \)-homeomorphism \( f_{\pl} \) that is \( \epsilon \)-close to \( f \),
\[ d(f(x), f_{\pl}(x)) < \epsilon \quad \mbox{ for all } \quad x \in M . \]

Shikata \cite{Shikata1966--DistanceFunctionSetDifferentiableStructures} has shown a similar result for bilipschitz homeomorphisms.
Our argument is similar to that in \cite{Shikata1966--DistanceFunctionSetDifferentiableStructures}, but generalized to quasiconformal mappings.
The point in both cases is that there is a uniform bound (in each dimension) on the quality of simplices that are required to triangulate a smooth Riemannian manifold.
One then shows that maps of sufficiently small dilatation (isometric or conformal) preserve simplex quality, enough that small simplicial approximations can be shown to be homeomorphisms.

In any case, the approximability by \( \pl \)-homeomorphisms is only non-trivial in high dimensions;
in dimensions up to \( 3 \) every homeomorphism can be approximated by a \( \pl \)-homeomorphism, a result going back to Moise \cite{Moise1952--AffineStructuresManifoldsIV}.
But the Theorem has real content in higher dimensions, where there do exist homeomorphisms that cannot be approximated by \( \pl \)-homeomorphisms.

\begin{theorem}\label{thm: non pl approximable homeo}
  In each dimension \( d \ge 5 \) there exists a smooth manifold \( M \), an orientation preserving homeomorphism \( f \colon M \to M \), and an \( \epsilon > 0 \) so that no \( \pl \)-homeomorphism is \( \epsilon \)-close to \( f \).
  The homeomorphism \( f \) may be chosen to be quasiconformal.
\end{theorem}
This Theorem follows by combining several classical results in differential topology, recalled in Section \ref{sec: non plizable mappings}.

\subsection{Indecomposable quasiconformal mappings}


It is easy to construct indecomposable maps by combining Theorems \ref{thm: pl approximation of nearly conformal} and \ref{thm: non pl approximable homeo}.

\begin{corollary}\label{thm-cor non factorable qc homeos}
  There exists a manifold \( M \) and a quasiconformal homeomorphism \( f \colon M \to M \) such that \( f \) does not admit a \( K_{\pl} \)-small factorization relative to any Riemannian metric on \( M \).
\end{corollary}
\begin{proof}
  Let \( f \colon M \to M \) be any quasiconformal homeomorphism that is not \( \pl \)-approximable, as in Theorem \ref{thm: non pl approximable homeo}, and let \( \epsilon > 0 \) be small enough that every \( \pl \)-homeomorphism of \( M \) is at least \( \epsilon \) away from \( f \).

  Because \( M \) is compact, the \( C^{0} \)-metric on homeomorphisms, given by
  \[ d(f, g) = \sup\{ d(f(x), g(x)) \colon x \in M \} , \]
  defines a topological group structure on \( \Homeo(M) \).
  This means in particular that the multiplication map is continuous, and so are the higher order multiplication maps (\( r \)-fold composition)
  \[ \mu_{r} \colon \Homeo(M)^{r} \to \Homeo(M) . \]
  By continuity, there exists a \( \delta > 0 \) small enough that any two lists of \( r \) elements \( (f_{1}, \ldots, f_{r}) \) and \( (f'_{1}, \ldots, f'_{r}) \) such that each \( d(f_{i}, f'_{i}) < \delta \), one has
  \[ d(f_{r} \composed \cdots \composed f_{1}, f'_{r} \composed \cdots \composed f'_{1}) < \epsilon . \]

  Now suppose that \( f \) admitted a \( K_{\pl} \)-small factorization
  \[ f = f_{r} \composed \cdots \composed f_{1} . \]
  By Theorem \ref{thm: pl approximation of nearly conformal}, there would exist for each \( f_{i} \) a \( \pl \)-homeomorphism \( f'_{i} \) that is \( \delta \)-close to \( f_{i} \).
  The composition of the \( f'_{i} \) would thus be a \( \pl \)-homeomorphism that is \( \epsilon \)-close to \( f \), a contradiction.
\end{proof}

\subsection{Structure of paper}

The first two sections of this paper are background, with Section \ref{sec: background qc mappings} on quasiconformal mappings and Section \ref{sec: background pl mappings} on piecewise linear mappings.

Section \ref{sec: preservation of shapes} concerns finite configurations of points in \( \R^{d} \) up to conformal action.
The key Lemma is \ref{thm-lem small dilatation preserves shapes}, showing that maps of small dilatation don't distort the shape of small configurations too much.
More precisely, given two disjoint compact subsets of configurations, there exists a sufficiently small dilatation \( K \) below which the image of small configurations in one set remain disjoint from the other.
An immediate application is to the case of \( d \)-simplices;
fixing a set of positively oriented simplices bounded away from being degenerate, the Lemma shows that there exists a dilatation below which all such simplices in \( \R^{d} \) are sent to positively oriented simplices.

To apply this Lemma, Section \ref{sec: full triangulations} covers a particular measure of non-degeneracy for simplices, the fullness.
For simplices in a Riemannian manifold there are two natural definitions of fullness, which need be reconciled.
In any case, in each dimension, there exists a lower bound on fullness \( \Theta_{d} \) such that every closed orientable Riemannian manifold admits arbitrarily fine triangulations with all simplices of fullness at least \( \Theta_{d} \).
This statement is used in Section \ref{sec: pl approximation of nearly conformal} to prove Theorem \ref{thm: pl approximation of nearly conformal}.
Finally, Section \ref{sec: non plizable mappings} is devoted to the proof of Theorem \ref{thm: non pl approximable homeo}.

\subsection{Discussion}

The indecomposability result obtained here is global in nature, and it remains to understand the local description.
For example, it remains open the question of whether there exist small factorizations of every quasiconformal mapping from the open ball into \( \R^{d} \).
The \( \pl \)-approximation statements here can, with more work, be extended to local maps of small dilatation, and it may well be that an assumption of local factorizability would lead to a similar contradiction as here.

To briefly expand on this, it is understood \cites{Sullivan1979--HyperbolicGeometryHomeomorphisms, Tukia-Vaeisaelae1981--LipschitzQuasiconformalApproximationExtension} that every topological manifold admits a quasiconformal structure, with charts whose transition functions are quasiconformal.
Define a potentially restricted class of manifolds, \emph{tame quasiconformal}, whose transitions are tame quasiconformal mappings.
Based on the results here, it is reasonable to make the following conjecture.
\begin{conjecture}
  The category of tame quasiconformal manifolds is the same as the category of \( \pl \)-manifolds.
\end{conjecture}
In other words, one might use the local small dilatation decompositions to `clean up' the transition functions, approximating chart by chart to a \( \pl \)-structure on each tame quasiconformal manifold.

If the conjecture held, it would mean in particular that there exist quasiconformal but non-tame manifolds, whose transition functions would demonstrate local indecomposable quasiconformal mappings.
However, we leave the pursuit of this idea to future work.

\subsection{Acknowledgements}
We would like to acknowledge Dennis Sullivan and Michael Freedman for many helpful discussions on the subject of this paper.
In particular, the paper would not have happened without Sullivan's suggestion for the overall approach here.

No AI tools were used in the preparation or writing of this article, with the exception of a last-pass spelling and grammar check.

\section{Quasiconformal mappings}\label{sec: background qc mappings}
\subsection{Quasiconformal mappings}
Fix two Riemannian manifolds \( M, N \), with distance metrics \( d_{M}, d_{N} \).
\begin{definition}
  For a homeomorphism \( f \colon M \to N \),
  the \emph{conformal distortion} of \( f \) at \( x \in M \) is
  \begin{equation}\label{eq: defn of pointwise linear dilatation}
    K(f,x) := \limsup_{r\to 0} \frac{\sup\bigl\{ d_{N} \bigl(f(x),f(y)\bigr) \bigm| y \in M \mbox{ s.t. } {d_{M}(x,y)\le r} \bigr\}}{\inf\bigl\{ d_{N} \bigl(f(x),f(y)\bigr) \bigm| y \in M \mbox{ s.t. } {d_{M}(x,y) \ge r} \bigr\}} .
  \end{equation}
  Say that \( f \) is \emph{quasiconformal} if there exists a constant uniformly bounding \( K(f,x) \) for all \( x \in M \).

  The \emph{linear dilatation} \( K(f) \) of a quasiconformal map \( f \) is given by the essential supremum of \( K(f,x) \) across the domain of \( f \).
\end{definition}
The minimum possible dilatation is 1, attained by conformal mappings.
A map \( f \) with dilatation close to 1 is nearly conformal, and sends the infinitesimal spheres about each point \( x \) into thin infinitesimal annuli about \( f(x) \).

There are several well-known alternative definitions of conformal dilatation, \cites{Vaisala1971--LecturesDimensionalQuasiconformalMappings, Gehring1962--RingsQuasiconformalMappingsSpace, Gehring-Martin-Palka2017--IntroductionTheoryHigherDimensionalQuasiconformalMappings}, which are generally commensurate.
For each definition, the composition of quasiconformal maps \( f, g \) is again quasiconformal, and the dilatation is submultiplicative, as in Equation \eqref{eq: submultiplicativity of dilatation}.
In particular, if \( g \) has dilatation \( 1 \), then its inverse does too, and \( K(f \composed g) = K(f) \) and \( K(g \composed f) = K(f) \).

\subsection{Bilipschitz mappings}
There is a stronger condition, corresponding to isometries.
\begin{definition}
  A map \( f \colon M \to N \) between metric manifolds is \emph{bilipschitz}, with \emph{(isometric) dilatation} \( K \), if
  \[ K^{-1}d_{M}(x,y) \le d_{N}(f(x), f(y)) \le K d_{M}(x,y) \quad \mbox{ for all } \quad x, y \in M . \]
\end{definition}
A \( K \)-bilipschitz mapping is immediately seen to be \( K^{2} \)-quasiconformal.
As a typical example of bilipschitz mappings, recall the following standard result from Riemannian geometry.
\begin{lemma}\label{thm: lem small exponentials are near isometric}
  For any dilatation \( K > 1 \), there exists a radius \( \delta > 0 \) so that for every point \( x \in M \), the exponential map \( \exp_{x} \colon T_{x} M \to M \) is \( \sqrt{K} \)-bilipschitz when restricted to the ball of radius \( \delta \) about the origin.
  The restriction is thus \( K \)-quasiconformal.
\end{lemma}
This holds because the exponential map is a diffeomorphism near the origin and an isometry at the origin.

\subsection{M\"obius transformations}
Let \( \compactR^{d} \) denote the one point compactification of \( \R^{d} \), equipped with the spherical metric.
The group of conformal transformations of \( \compactR^{d} \)---the \emph{M\"obius group}---is generated by translations, scalings, rotations fixing the origin and sphere inversions in \( \R^{d} \), each extended to \( \compactR^{d} \) in the obvious manner.
In particular, the orientation preserving M\"obius transformations of \( \compactR^{d} \) that fix the point at infinity form the group of \emph{similarities} of \( \R^{d} \): compositions of translations, scalings and rotations fixing the origin.

\subsection{Compactness of quasiconformal mappings}

We will need the following theorem on compactness of the family of quasiconformal mappings.
Let \( D \) be an open domain in \( \R^{d} \).
A sequence of continuous functions \( f_{i} \colon D \to \R^{d} \) converges \emph{locally uniformly} to \( f \) if for each compact set \( K \subset D \), the sequence of restrictions \( f_{i}|_{K} \) converges uniformly to \( f|_{K} \).
The next Theorem follows, for example, from Theorem 6.6.26 of \cite{Gehring-Martin-Palka2017--IntroductionTheoryHigherDimensionalQuasiconformalMappings}.
\begin{theorem}\label{thm-convergence to a Mobius}
  Suppose given a sequence of quasiconformal mappings
  \[ f_{i} \colon B(0,R) \to \R^{d} \]
  such that \( K(f_{i}) \to 1 \).
  If there are two distinct points \( x_{1}, x_{2} \in B(0,R) \) and corresponding \( y_{1}, y_{2} \in \R^{d} \) so that each \( f_{i} \) sends \( x_{a} \) to \( y_{a} \) (with \( a = 1,2 \)), then the sequence \( f_{i} \) has a subsequence that converges locally uniformly to the restriction of a M\"obius transformation that sends each \( x_{a} \) to \( y_{a} \).
\end{theorem}

\section{Piecewise linear mappings}\label{sec: background pl mappings}
\subsection{Triangulations and the PL-category}\label{sec-PL category}

We recall the basic notions of the piecewise linear category.
One may refer to \cites{Hudson1969--PiecewiseLinearTopology, Rourke-Sanderson1972--IntroductionPiecewiseLinearTopology} for further details.


For \( d' \le d \), a \emph{\( d' \)-simplex} in \( \R^{d} \) is a map \( \sigma \colon \{ 0, 1, \ldots, d' \} \to \R^{d} \).
The \emph{vertices} of \( \sigma \) are the values \( \sigma(0), \ldots, \sigma(d') \).
It is convenient to allow repetition of vertices, in which case the simplex is clearly degenerate.
More generally, a simplex is \emph{degenerate} if the set of \emph{base vectors} \( \sigma(0)\sigma(i) \) for \( i > 0 \) comprises a linearly dependent set, and \emph{non-degenerate} otherwise.
When the simplex is top dimensional, \( d' = d \), we say that the simplex is \emph{positively oriented} in case these base vectors are a positively oriented basis of \( \R^{d} \).

For a non-degenerate simplex \( \sigma \), the \emph{image} or \emph{geometric realization} \( |\sigma| \) is the convex hull of its vertices.
It is clear that the faces of \( |\sigma| \) are the images of the subsimplices of \( \sigma \).
A subset \( P \subseteq \R^{d} \) is a \emph{polyhedron} if it admits a locally finite triangulation by the geometric realization of non-degenerate simplices.


Given a polyhedron \( P \subset \R^{d} \), a continuous mapping \( f \colon P \to \R^{m} \) is \emph{piecewise linear}, or \( \pl \), if there exists a triangulation \( T \) of \( P \) such that the restriction to each simplex is affine linear.
Say that \( f \) is a \emph{\( \pl \)-homeomorphism} if it has a \( \pl \) inverse from its image, which holds already if \( f \) is \( \pl \) and a topological homeomorphism.
It may be worth remarking that the triangulation cannot be fixed in advance with respect to the set of all \( \pl \)-maps of \( P \), although arbitrary subdivisions of a given triangulation typically suffice.
For example, given a second \( \pl \) map \( g \), the composition \( g \composed f \) is again \( \pl \), but may only be affine linear on a finer triangulation than \( T \).

A \emph{\( \pl \)-manifold} \( M \) is a (topological) manifold equipped with an atlas such that transitions between charts are \( \pl \)-homeomorphisms.
Every polyhedron \( P \) that is locally \( \pl \)-homeomorphic to \( \R^{d} \) is a \( \pl \)-manifold, and every \( \pl \)-manifold is homeomorphic to such a polyhedron.
This allows the following definition:
a \emph{triangulation} of \( M \) is a triangulation on \( P \) and homeomorphism \( P \to M \).
It will be convenient to use the notation \( T \) for a simplicial complex and \( |T| \) its underlying polyhedron.

A \( \pl \)-homeomorphism between two \( \pl \)-manifolds is a homeomorphism \( f \colon M \to N \) that is \( \pl \) in charts.
This is equivalent to the existence of triangulations of \( M \) and \( N \) on which \( f \) is \( \pl \).

\subsection{Smoothly compatible triangulations}
To relate \( \pl \)-structures and smooth structures, one needs the notion of a \emph{piecewise differentiable map}.
Given a \( \pl \)-manifold \( M \) and a smooth manifold \( N \), a piecewise differentiable map is \( f \colon M \to N \) such that there exists a triangulation \( |T| \to M \) so that the restriction to each simplex of \( |T| \to M \to N \) is smooth.
Such \( f \) is a \emph{piecewise differentiable homeomorphism} if it is a homeomorphism and the restriction to each simplex is a smooth embedding.
In case of a piecewise differentiable homeomorphism \( f \colon |T| \to N \), the triangulation \( T \) is \emph{compatible} with the smooth structure of \( N \).

Whitehead \cite{Whitehead1940--C1Complexes} showed that every differentiable manifold \( N \) admits a compatible \( \pl \)-manifold structure, and the structure is essentially unique.
More precisely, there exists a \( \pl \)-manifold \( M \), a piecewise differentiable homeomorphism \( M \to N \), and for any other compatible \( M' \) there is a \( \pl \)-homeomorphism \( M \to M' \).

If a manifold is smooth, we always assume a given \( \pl \)-structure on it is a compatible Whitehead structure.
In particular, while diffeomorphisms are not themselves \( \pl \), they respect the underlying \( \pl \)-structures:
for \( |T| \to M \) a compatible \( \pl \)-structure, and diffeomorphism \( f \colon M \to N \), the pushforward (composition) \( |T| \to N \) is also compatible.
More generally, a homeomorphism \( f \colon M \to N \) between smooth manifolds can be said to be \( \pl \) if it pushes forward a compatible \( \pl \)-structure of \( M \) to a compatible structure of \( N \).
Equivalently, \( f \) is \( \pl \) if there exists a triangulation of \( M \) along which \( f \) is a piecewise differentiable homeomorphism.


We close the subsection with a characterization of \( \pl \)-homeomorphisms between smooth manifolds.
Recall, for oriented closed manifolds \( M, N \) of equal dimension \( d \), that each continuous map \( f \colon M \to N \) induces a map on cohomology groups, \( f^{*} \colon H^{d}(N) \to H^{d}(M) \).
The map \( f^{*} \) is multiplication by an integer, \( \deg(f) \), called the \emph{degree} of \( f \).

Supposing that \( M, N \) are smooth and \( f \) is piecewise differentiable, the degree can be computed locally.
Say that \( x \in M \) is \emph{regular} where \( \d_{x} f \) is defined and invertible, and define \( \deg(f,x) \) as \( 1 \) or \( -1 \) in case \( \d_{x} f \) is orientation preserving or reversing.
Sard's theorem holds equally well for piecewise differentiable maps as for smooth,
and the degree of \( f \) can be computed at any regular value \( y \in N \) by
\[ \deg(f) = \sum_{x \in f^{-1}(y)} \deg(f,x) . \]
The following is thus clear.
\begin{lemma}\label{thm-lem locally degree 1 map is homeo}
  If \( f \colon M \to N \) is a mapping of degree \( 1 \), piecewise differentiable relative an oriented triangulation \( |T| \to M \), and sends each \( d \)-simplex diffeomorphically with positive orientation, then \( f \) is a \( \pl \)-homeomorphism.
\end{lemma}

\subsection{Embeddings and compatible metrics}
It is well known \cite{Whitney1936--DifferentiableManifolds} that a smooth manifold \( M \) embeds smoothly in high-dimensional \( \R^{m} \), uniquely up to smooth isotopy.
We will assume any given Riemannian metric on \( M \) is compatible with a particular smooth embedding.
This is no loss of generality, due to the Nash embedding Theorem \cite{Nash1956--ImbeddingProblemRiemannianManifolds}.
For any smooth Riemannian \( M \), we will fix an isometric embedding once and for all;
the specific choice of embedding will not affect our qualitative arguments.

We will also fix a choice of smooth normal bundle, the points \( M_{\epsilon} \) within small \( \epsilon > 0 \) of \( M \).
For small enough \( \epsilon \), the space \( M_{\epsilon} \) is a smooth \( (m-\dim M) \)-ball bundle over \( M \), with normal projection
\[ \nu \colon M_{\epsilon} \to M \]
sending each point to its nearest point in \( M \).

\subsection{Secant approximations}
Given a triangulated manifold \( (M, T) \) and a continuous map \( f \colon M \to \R^{m} \), the \emph{ambient secant approximation of \( f \) along \( T \)} is the unique \( \pl \) map that sends each vertex \( v \) of \( T \) to \( f(v) \) and is then extended linearly across each simplex of \( T \).

Now suppose given continuous \( f \colon M \to N \) to a smooth manifold \( N \).
Fix isometric embedding \( \iota \colon N \to \R^{n} \) and normal bundle \( \nu \colon N_{\epsilon} \to N \).
For any sufficiently fine triangulation \( T \), the ambient secant approximation \( F_{T} \) of \( T \) along \( \iota \composed f \) will be such that the image of \( F_{T} \) is contained in \( N_{\epsilon} \).
In such case, the \emph{secant approximation} \( f_{T} = \nu \composed F_{T} \) is a piecewise differentiable mapping \( M \to N \).

Suppose \( T \) is yet finer, so that for each \( x \in M \), the straight line between \( F_{T}(x) \) and \( \iota \composed f(x) \) is contained in \( N_{\epsilon} \).
In such case, the secant approximation is clearly homotopic to \( f \).
In particular, if \( f \) is an orientation preserving homeomorphism, then \( f_{T} \) has degree 1.

Note that even if \( f \) is a homeomorphism, the secant approximation may not be, regardless of how fine the triangulation \( T \) is taken.
In general, more control over image simplices is needed, such as is provided by bounds on dilatation of \( f \).

On the other hand, secant approximations can be made arbitrarily close to \( f \) by the following standard proposition.
\begin{proposition}
  For continuous \( f \colon M \to N \) and arbitrary \( \epsilon > 0 \), there exists \( \delta > 0 \) small enough that the secant approximation of \( f \) along any \( \delta \)-fine triangulation is \( \epsilon \)-close to \( f \).
\end{proposition}


Another upshot of the secant approximations is that collections of points in \( M \) of small enough diameter unambiguously determine simplices in \( M \).
By mild abuse of notation, we will say for \( d' \le \dim(M) \) that any function \( \sigma \colon \{0,1,\ldots,d'\} \to M \) is a \emph{\( d' \)-simplex} in \( M \).
Say that \( \sigma \) is \emph{small} if the ambient simplex \( \iota \composed \sigma \) is contained in the normal bundle \( M_{\epsilon} \).
In this case, the geometric realization of \( \sigma \) in \( M \) is the normal projection of \( |\iota \composed \sigma| \).

\section{Almost conformal mappings and shapes}\label{sec: preservation of shapes}

We consider configurations of points in \( \R^{d} \), and the extent to which a quasiconformal map can distort the shape of these configurations.
While the general M\"obius transformation of \( \compactR^{d} \) can distort shape fairly arbitrarily, the situation is better for transformations whose singular point is far from the origin.
By rescaling, it is the same to say that M\"obius transformations fairly well preserve the shape of small configurations far from the singular point.
Using this, and the fact that the set of (not fully degenerate) shapes is compact, it will be shown that maps of small enough dilatation also preserve the shapes of small configurations far from the boundary of the domain.

A particular application is that small dilatation maps don't reverse the orientation of small simplices that are bounded away from being degenerate.

\subsection{Shapes}
The space of configurations of \( r \) points in \( \R^{d} \) (ordered, and repetition allowed) is identified with \( \bigl( \R^{d} \bigr)^{r} \).
The action of orientation preserving similarities on \( \R^{d} \) induces the product action on \( r \)-configurations,
\[ g \cdot (v_{1}, \ldots, v_{r}) = (g \cdot v_{1}, \ldots, g \cdot v_{r}) \quad \mbox{ for } \quad g \in G, \]
\( G \) the group of similarities of \( \R^{d} \).
Within the \( r \)-configurations is the set of fully degenerate configurations,
\[ \mathcal{D}^{r} := \left\{ (v,\ldots,v) \in \bigl(\R^{d}\bigr)^{r} \mid v \in \R^{d} \right\} , \]
which we will exclude from consideration.
It is easy to see that the action of similarities preserves the set of fully degenerate configurations, and so also its complement.

\begin{definition}
  The \emph{space of \( r \)-shapes} \( \shapes^{r} \) is the quotient of $\bigl(\R^{d}\bigr)^{r} - \mathcal{D}^{r}$ by the action of orientation preserving similarities of $\R^{d}$.

  For a point $x \in \shapes^{r}$, a configuration \( p \in \bigl( \R^{d} \bigr)^{r} \) \emph{has shape} $x$ if it is in the equivalence class $x$, i.e. \( x = [p] \).
\end{definition}

An important special case is simplices up to similarity, by taking the vertices as a configuration;
although the shape of this configuration loses the order of vertices, the orientation of a non-degenerate \( (d+1) \)-shape in \( \R^{d} \) is well defined.
We use \( \orientedShapes^{d+1} \) to denote the open subset of \( \shapes^{d+1} \) corresponding to positively oriented \( d \)-simplices.

The following Lemma is easy but crucial.
\begin{lemma}
  For each \( r \), the space of shapes \( \shapes^{r} \) is compact.
\end{lemma}
\begin{proof}
Consider the set \( C \) of not fully degenerate \( r \)-configurations in \( \R^{d} \) with center of mass at the origin and diameter equal to 1.
The set \( C \) is closed and bounded in \( \bigl(\R^{d}\bigr)^{r} \), so is compact.
Any configuration not in \( \mathcal{D}^{r} \) is equivalent to an element of \( C \) by a similarity, so \( C \) surjects onto \( \shapes^{r} \), which is thus compact.
\end{proof}

We make the following notational convention.
Suppose given an open embedding \( f \colon U \to \R^{d} \) and a subset of shapes \( S \subset \shapes^{r} \).
The image of shapes by \( f \) is the set
\[ f[S] := \left\{ [f(p)] \in \shapes^{r} \colon p \in U^{r} \mbox{ such that } [p] \in S \right\} . \]
Given compact \( C \subset U \),
the image of shapes in \( C \) by \( f \) is the set
\[ f_{C}[S] := \left\{ [f(p)] \in \shapes^{r} \colon p \in C^{r} \mbox{ such that } [p] \in S \right\} . \]
For \( \delta > 0 \), the image of \( \delta \)-small shapes in \( C \) by \( f \) is
\[ f_{C, \delta}[S] := \left\{ [f(p)] \in \shapes^{r} \colon p \in C^{r} \mbox{ such that } [p] \in S \mbox{ and } \diam(p) < \delta \right\} . \]

\subsection{M\"obius maps preserve shape of small configurations}
We first demonstrate that any M\"obius transformation preserves the shape of configurations far from its singular point.
\begin{lemma}\label{thm-lem Mob inversions close to infinity preserve shape}
  Fix two disjoint closed sets of \( r \)-shapes \( S, S' \subset \shapes^{r} \) in \( \R^{d} \).
  \begin{enumerate}
  \item\label{part: large radial Mobius preservation}
    There exists a radius \( R = R_{S,S'} \) such that every orientation preserving M\"obius transformation \( \mu \colon B(0,R) \to \R^{d} \) satisfies
    \[ (\mu_{B(0,1)}[S]) \cap S' = \varnothing . \]
    In other words, \( \mu \) does not send any configuration contained in \( B(0,1) \) and with shape in $S$ to a configuration with shape in \( S' \).

  \item
    For any fixed open \( U \subseteq \R^{d} \) and compact \( C \subset U \), there exists a diameter \( \delta > 0 \) such that every orientation preserving M\"obius transformation \( \mu \colon U \to \R^{d} \) satisfies
    \[ (\mu_{C,\delta}[S]) \cap S' = \varnothing . \]
    In other words, \( \mu \) does not send any configuration in \( C \) of diameter less than \( \delta \) and shape in $S$ to a configuration with shape in \( S' \).
  \end{enumerate}
\end{lemma}

\begin{proof}
  To prove \eqref{part: large radial Mobius preservation}, suppose that no bound on radius held, so that there existed a sequence \( (\mu_{i}, R_{i}, \sigma_{i}, \sigma'_{i}) \) such that
  \begin{itemize}
  \item
    each \( \mu_{i} \) is a M\"obius mapping \( \mu_{i} \colon B(0,R_{i}) \to \R^{d} \),

  \item
    the radii \( R_{i} \) converge to infinity,

  \item
    each \( \sigma_{i} \) is a configuration contained in \( B(0,1) \) and of shape in \( S \),

  \item
    each \( \sigma_{i}' \) is a configuration of shape in \( S' \), and

  \item
    each \( \mu_{i} \) sends \( \sigma_{i} \) to \( \sigma'_{i} \).
  \end{itemize}
  Given compactness of \( S \), we may take a subsequence to guarantee that the shapes \( [\sigma_{i}] \) converge to a fixed shape \( [\sigma] \in S \) and the \( [\sigma'_{i}] \) converge to a fixed \( [\sigma'] \in S' \).

  Fix any representative \( \sigma \subset B(0,1) \) for \( [\sigma] \), which configuration will have finite positive diameter.
  By precomposing each \( \mu_{i} \) with a similarity, we may assume that \( \sigma_{i} \to \sigma \) as configurations;
  care is needed here that in doing so the radii \( R_{i} \) may decrease, but the decrease is at worst scaling by a factor dependent on the diameter of \( \sigma \), so the \( R_{i} \) still converge to infinity.
  Next, by post-composing with similarities, we arrange that \( \sigma'_{i} \to \sigma' \) for some fixed representative \( \sigma' \) of \( [\sigma'] \).

  We have without loss of generality a sequence \( (\mu_{i}, R_{i}, \sigma_{i}, \sigma'_{i}) \) as above for which furthermore \( \sigma_{i} \to \sigma \) and \( \sigma'_{i} \to \sigma' \).
  Because \( \sigma \) is not fully degenerate, it has two distinct vertices (say \( \sigma(0) \) and \( \sigma(1) \)).
  By composing with similarities near the identity, we may suppose each \( \sigma_{i} \) to have the same vertices, \( \sigma_{i}(0) = \sigma(0) \) and \( \sigma_{i}(1) = \sigma(1) \), and that each \( \sigma'_{i} \) has the same (distinct) first two vertices of \( \sigma' \).
  The sequence of M\"obius maps \( \mu_{i} \) separates the three points \( \sigma(0), \sigma(1), \infty \) in \( \compactR^{d} \), and so there exists a subsequence converging to a M\"obius mapping \( \mu \).
  The limiting \( \mu \) fixes the point at \( \infty \), so must be a similarity, but also maps \( \sigma \) to \( \sigma' \), a contradiction.

  The statement (2) follows from (1).
  Let \( \delta_{0} > 0 \) be the distance from \( C \) to the complement \( U^{c} \subset \R^{d} \) and \( R = R_{S,S'} \) the radius guaranteed by statement \eqref{part: large radial Mobius preservation}.
  Then any positive diameter \( \delta < \delta_{0} / R \) will suffice.
  Given M\"obius \( \mu \colon U \to \R^{d} \) and configuration \( \sigma \subset C \) of diameter less than \( \delta \), we may suppose after a translation that \( \sigma \subset B(0,\delta) \) and that \( \mu \) is defined on \( B(0,\delta_{0}) \).
  But, by precomposing \( \mu \) with scaling by \( R / \delta_{0} \), the hypothesis of \eqref{part: large radial Mobius preservation} can be made to hold, and thus if \( [\sigma] \in S \) it cannot hold that \( [\mu(\sigma)] \in S'  \).
\end{proof}

\subsection{Nearly conformal maps preserve the shape of small configurations}
The conclusion of Lemma \ref{thm-lem Mob inversions close to infinity preserve shape} can be extended to maps of sufficiently small dilatation.

\begin{lemma}\label{thm-lem small dilatation preserves shapes}
  Fix two disjoint closed sets of \( r \)-shapes \( S, S' \subset \shapes^{r} \) in \( \R^{d} \).

  \begin{enumerate}
  \item\label{part: large radial small QC preservation}
    For any radius \( R > R_{S,S'} \), there is a dilatation $K > 1$ such that any orientation preserving \( K \)-quasiconformal mapping $f \colon B(0,R) \to \R^{d}$ satisfies
    \[ (f_{B(0,1)}[S]) \cap S' = \varnothing . \]

  \item
    Fix \( K \) as in \eqref{part: large radial small QC preservation}.
    For every open \( U \subseteq \R^{d} \) and compact \( C \subset U \), there exists a diameter \( \delta > 0 \) such that every \( K \)-quasiconformal \( f \colon U \to \R^{d} \) satisfies
    \[ (f_{C,\delta}[S]) \cap S' = \varnothing . \]
  \end{enumerate}
\end{lemma}

\begin{proof}
  To prove the first statement, suppose for contradiction a sequence \( (f_{i}, \sigma_{i}, \sigma_{i}') \), where
  \begin{itemize}
    \item each \( f_{i} \colon B(0,R) \to \R^{d} \) is a quasiconformal mapping,

    \item the dilatations of the \( f_{i} \) converge to \( 1 \),

    \item each \( \sigma_{i} \) is a configuration contained in \( B(0,1) \) and of shape in \( S \),

    \item each \( \sigma_{i}' \) is a configuration of shape in \( S' \), and

    \item each \( f_{i} \) sends \( \sigma_{i} \) to \( \sigma_{i}' \).
  \end{itemize}
  By the same argument as for Lemma \ref{thm-lem Mob inversions close to infinity preserve shape}, we may assume convergence of configurations \( \sigma_{i} \to \sigma \) and \( \sigma'_{i} \to \sigma' \) of shape in \( S \) and \( S' \) respectively.
  The one new point is that the precomposition of each \( f_{i} \) by a similarity should be chosen so that the radius of definition does not decrease below \( R_{S,S'} \).
  This can be accomplished by choosing the representative configuration \( \sigma \) of sufficiently large diameter but still in \( B(0,1) \).
  In doing so, it may be assumed that there is \( R' > R_{S,S'} \) for which each \( f_{i} \) is defined on \( B(0,R') \).

  We have arranged matters so that Theorem \ref{thm-convergence to a Mobius} applies, and some subsequence \( f_{i} \) converges to a M\"obius transformation that maps \( B(0,R') \) to \( \R^{d} \) and sends \( \sigma \) to \( \sigma' \), contradicting Lemma \ref{thm-lem Mob inversions close to infinity preserve shape}.

  The second statement follows from the first by precisely the same argument as in \ref{thm-lem Mob inversions close to infinity preserve shape}.
\end{proof}

\subsection{Numerics}
The argument given is by compactness, so provides no quantitative information on what dilatation is required to satisfy the conclusion.
It would of course be interesting to have more quantitative statements, in particular on the dilatation required to preserve orientation of simplices.
Quantitative bounds are not difficult to obtain in the Lipschitz setting, by direct geometric argument---how much stretch does one need to invert a simplex?---and these arguments could be combined with the well-developed Mori-distortion type theorems on quasiconformal mappings.
We don't pursue this further here; only the existence of critical dilatations is needed for the application at hand.

\subsection{The distortion of shapes}
The lemma can be interpreted as follows: take a closed subset of shapes \( S \subset \shapes^{r} \), and an open set \( V \) containing \( S \).
Since \( S \) and the complement of \( V \) are disjoint and closed, the lemma applies, and we conclude that maps of small dilatation and large radius send configurations with shape in \( S \) to configurations of shape close to \( S \).

\subsection{Control on simplex-wise distortion}
We conclude this section with a Lemma demonstrating that the distortion of small simplices is controlled for maps of small enough dilatation.
This Lemma is used to demonstrate that the dilatation of \( \pl \)-approximations can be controlled.

For two non-degenerate \( d \)-simplices \( \sigma, \sigma' \) in \( \R^{d} \), let \( A_{\sigma',\sigma} \) be the unique affine map sending \( \sigma \) to \( \sigma' \).
Define the \emph{affine ellipticity} of any simplex \( \sigma \) relative a fixed positive simplex \( \sigma' \) by the continuous function
\[ \begin{tikzcd}[row sep={7mm,between origins}, column sep={4em}]
    E'_{\sigma'}: \bigl(\R^{d}\bigr)^{d+1} \backslash \mathcal{D}^{d+1} \ar[r] & {[1, \infty]} & \\
    \sigma \ar[r, mapsto] & K(A_{\sigma,\sigma'}) & \mbox{if } [\sigma] \in \orientedShapes^{d+1} \\
    \sigma \ar[r, mapsto] & \infty & \mbox{if } [\sigma] \in \shapes^{d+1} - \orientedShapes^{d+1} .
\end{tikzcd} \]
This function is invariant under the action of similarities, so factors through a continuous function
\[ E_{\sigma'} \colon \shapes^{d+1} \xrightarrow{\qquad} [1,\infty] . \]
Moreover, if \( \sigma'' \) and \( \sigma' \) have the same shape, then \( E_{\sigma''} = E_{\sigma'} \), so there is no ambiguity in defining \( E_{[\sigma']} = E_{\sigma'} \), where \( \sigma' \) is any representative of the shape \( [\sigma'] \).

\begin{lemma}\label{thm: lem control on affine distortion}
  Fix a compact \( S \subset \orientedShapes^{d+1} \) and a dilatation \( K' > 1 \).
  There exists a dilatation \( K \) so that every compact-open pair \( C \subset U \subseteq \R^{d} \) admits a radius \( \delta > 0 \) small enough that for every \( K \)-quasiconformal map \( f \colon U \to \R^{d} \),
  and for every \( \delta \)-small \( d \)-simplex \( \sigma \) contained in \( C \) and of shape in \( S \), the simplex \( f(\sigma) \) is positively oriented and \( K(A_{f(\sigma),\sigma}) < K' \).
\end{lemma}
\begin{proof}
  Suppose first that \( S \) is small enough that for some \( x = [\sigma] \in S \),
  the restriction of \( E_{x} \) to \( S \) is bounded above by \( \sqrt{K'} \).
  Let \( V = E_{x}^{-1}\bigl([1, \sqrt{K'})\bigr) \).
  Then \( V \) is open, contains \( S \), and has closure \( E_{x}^{-1}\bigl([1, \sqrt{K'}] \bigr) \), which is contained in \( \orientedShapes^{d+1} \).
  By application of Lemma \ref{thm-lem small dilatation preserves shapes}, there is a dilatation \( K > 1 \) and then for each compact-open pair \( C \subset U \) a diameter \( \delta > 0 \) so that every \( \delta \)-small simplex \( \sigma' \) in \( C \) of shape in \( S \) is sent by every \( K \)-quasiconformal \( f \colon U \to \R^{d} \) to a simplex \( \sigma'' \) of shape in \( V \).
  For every such \( \sigma' \), the dilatation bound holds,
  \[ K(A_{\sigma'', \sigma'}) \le K(A_{\sigma'',\sigma}) K(A_{\sigma,\sigma'}) < \sqrt{K'}\sqrt{K'} = K' . \]

  To prove the claim for arbitrary compact \( S \subset \orientedShapes^{d+1} \), cover \( S \) with the open sets \( S_{x} = E_{x}^{-1}\bigl([1, \sqrt{K'})\bigr) \) as \( x \) ranges over \( S \).
  There is a finite subcover and sufficient dilatations \( K_{i} \) for each open in the subcover, so take the minimum amongst the \( K_{i} \).
\end{proof}

\section{Full triangulations}\label{sec: full triangulations}
We apply the results of the previous section to triangulations of Riemannian manifolds.
\subsection{Fullness of Euclidean simplices}

The \emph{fullness} (cf. \cite{Whitney1957--GeometricIntegrationTheory}*{section IV.B}) of a non-degenerate \( d' \)-simplex \( \sigma \) in \( \R^{d} \), is given in terms of its convex hull \( |\sigma| \), by
\[ \Theta_{\R^{d}}(\sigma) = \frac{\vol(|\sigma|)}{(\diam(|\sigma|))^{d'}} . \]
In the case of a top dimensional \( d \)-simplex in \( \R^{d} \), this definition can be extended to include orientation,
\[ \Theta_{\R^{d}}(\sigma) = \sign(\sigma)\frac{\vol(|\sigma|)}{(\diam(|\sigma|))^{d}} , \]
with \( \sign(\sigma) = \pm 1 \) dependent on whether \( \sigma \) is positively or negatively oriented.

The function \( \Theta_{\R^{d}} \) is clearly similarity invariant, and so descends to a function on simplex shapes \( \shapes^{d+1} \).
The key point is that fullness is well controlled by maps of small dilatation.

\begin{lemma}\label{thm: lem small dilatation preserves fullness}
  Fix dimension \( d \), fullness \( \Theta > 0 \), and a constant \( \alpha \in (0,1) \).
  There exists a dilatation \( K > 1 \) such that for every open \( U \subseteq \R^{d} \) and compact \( C \subset U \), there exists a diameter \( \delta > 0 \) so that every \( K \)-quasiconformal \( f \colon U \to \R^{d} \) sends \( \delta \)-small simplices in \( C \) of fullness more than \( \Theta \) to simplices of fullness more than \( \alpha \Theta \).
\end{lemma}
\begin{proof}
  This follows from the continuity of \( \Theta_{\R^{d}} \) and Lemma \ref{thm-lem small dilatation preserves shapes} with
  \begin{align*}
    S & = \Theta^{-1}_{\R^{d}}([\Theta, \infty)) \\
    S' & = \Theta^{-1}_{\R^{d}}((-\infty, \alpha\Theta]) .
  \end{align*}
\end{proof}

\subsection{Fullness of Riemannian simplices}

Now fix a closed oriented Riemannian \( d \)-manifold \( M \).
For any simplex \( \sigma \in M \) of diameter below the injectivity radius, define the \emph{internal fullness} of \( \sigma \) by
\[ \Theta_{M}(\sigma) := \Theta_{\R^{d}}(\exp_{\sigma(0)}^{-1}(\sigma)) . \]
Here
\[ \exp_{\sigma(0)} \colon T_{\sigma(0)} M \to M \]
is the Riemannian exponential map at \( \sigma(0) \), and fullness of \( \exp_{\sigma(0)}^{-1}(\sigma) \) is the Euclidean fullness measured with respect to the Riemannian metric on \( T_{\sigma(0)} M \).

It is also useful to measure simplex quality relative to an isometric embedding \( \iota \colon M \to \R^{m} \);
denote the \emph{ambient fullness} of \( \sigma \in M \) by
\[ \Theta_{\iota}(\sigma) := \Theta_{\R^{m}}(\iota(\sigma)) . \]

For top dimensional \( d \)-simplices, the internal fullness can be extended to have sign depending on the orientation relative \( T_{\sigma(0)} M \) (as was done for fullness of Euclidean simplices.)
The situation is more subtle for ambient fullness, because ambient \( d \)-simplices in \( \R^{m} \) do not have an orientation.
However, the next lemma and its corollary show that the two notions of fullness are commensurate for sufficiently small simplices in \( M \); for these small-enough simplices we can also give the ambient fullness a sign depending on the projection to the tangent plane of \( M \) at \( \sigma(0) \).

\begin{lemma}\label{thm: lemma small secant to tangent projections bilipschitz}
  Fix a dilatation \( K > 1 \), a fullness \( \Theta > 0 \), and constant \( \alpha \in (0,1) \).
  There exists a diameter \( \delta > 0 \) such that for any \( \delta \)-small and \( \Theta \)-full (either of ambient or internal) \( d \)-simplex, the orthogonal projection from the secant plane through \( \sigma \) to the tangent plane \( T_{\sigma(0)} M \subset \R^{m} \) is \( K \)-quasiconformal.
\end{lemma}

\begin{proof}
  This is one of the key points in the proof of Whitehead's Triangulation Theorem, \cite{Whitney1957--GeometricIntegrationTheory}*{Lemma 14c}, albeit not stated in terms of quasiconformality.
  For any lower bound on fullness \( \Theta \), every small enough simplex \( \sigma \in M \subset \R^{m} \) has secant plane sufficiently close to the (ambient) tangent plane \( T_{\sigma(0)} M \) when measured in the metric of the Grassmannian \( \Gr_{d}(\R^{m}) \).
  Indeed, for any \( \epsilon > 0 \), the diameter \( \delta \) can be chosen small enough that the base vectors \( v_{i} = \sigma(0)\sigma(i) \) of any \( \delta \)-small simplex \( \sigma \) are very close to their tangent plane projections \( w_{i} \in T_{\sigma(0)} M \),
  \begin{equation}\label{eq: tangent projections bound}
    | v_{i} - w_{i} | < \epsilon | v_{i} | .
  \end{equation}
  If \( \sigma \) is \( \Theta \)-full, the basis \( v_{i} \) of the secant plane is bounded away from degeneracy, and the bound \eqref{eq: tangent projections bound} leads to a bound on the difference between the tangent and secant planes, which is in fact linear in \( \epsilon \).
  The orthogonal projection from a plane in \( \R^{m} \) to a nearby plane is nearly isometric, and thus nearly conformal.
  This suffices to prove the claim.

\end{proof}

The previous Lemma has the following Corollary, by a direct application of Lemma \ref{thm: lem small dilatation preserves fullness}.

\begin{corollary}\label{thm: cor internal ambient fullness commensurate}
  Fix \( M \), a fullness \( \Theta > 0 \) and constant \( \alpha \in (0,1) \).
  \begin{enumerate}
  \item
    There exists a diameter \( \delta > 0 \) such that any \( \delta \)-small \( d' \)-simplex of \( M \) with ambient fullness \( \Theta \) has internal fullness at least \( \alpha \Theta \).

  \item
    There exists a diameter \( \delta > 0 \) such that any \( \delta \)-small \( d' \)-simplex of \( M \) with internal fullness \( \Theta \) has ambient fullness at least \( \alpha \Theta \).
  \end{enumerate}
\end{corollary}

Continuing with the embedding \( M \to \R^{m} \), fix a small tubular neighborhood \( M_{\epsilon} \) of \( M \), as well as the normal projection
\[ \nu \colon M_{\epsilon} \to M . \]
The following is well understood, for example \cite{Whitney1957--GeometricIntegrationTheory}*{Lemma 21a}.
The point again is that sufficiently small full simplices on \( M \) are close to being tangent to \( M \).

\begin{lemma}\label{thm: lemma small normal projections}
  Fix a fullness \( \Theta > 0 \).
  There exists a diameter \( \delta > 0 \) such that for any \( \delta \)-small \( \Theta \)-full simplex \( \sigma \) in \( M \),

  \begin{enumerate}
  \item
    the convex hull of \( \sigma \) is contained in \( M_{\epsilon} \),

  \item
    the normal projection \( \nu \) restricted to \( \sigma \) is a smooth embedding \( \sigma \to M \), and

  \item
    the orientation of the embedding \( \sigma \to M \) is the same as the orientation of the orthogonal projection of \( \sigma \) to \( T_{\sigma(0)} M \).
  \end{enumerate}
\end{lemma}

\subsection{Triangulations of uniformly bounded fullness}

To end this section, we recall one further Lemma.
It is an essential part of the proof of the Whitney triangulation Theorem for smooth manifolds, and can be found for example in \cites{Shikata1966--DistanceFunctionSetDifferentiableStructures, Whitney1957--GeometricIntegrationTheory}.
\begin{lemma}\label{thm: lem fine quasielliptic triangulations}
  In each dimension \( d \), there exists \( \Theta_{d} > 0 \) such that every closed Riemannian \( d \)-manifold admits arbitrarily fine triangulations with all simplices of fullness at least \( \Theta_{d} \).
\end{lemma}

\section{Approximation of almost conformal mappings by PL-homeomorphisms}\label{sec: pl approximation of nearly conformal}

We can now prove the \( \pl \)-approximation Theorem \ref{thm: pl approximation of nearly conformal}, as stated in the introduction.
Recall the statement:
\emph{In each dimension \( d \) there is a conformal dilatation \( K_{\pl} > 1 \) such that every \( K_{\pl} \)-quasiconformal mapping \( f \colon M \to N \) between smooth closed Riemannian \( d \)-manifolds can be arbitrarily well approximated by \( \pl \)-homeomorphisms in the \( C^{0} \) function topology.
Furthermore, for each \( K' > 1 \), there exists \( K > 1 \) such that the \( \pl \)-approximation of any \( K \)-quasiconformal mapping can be chosen to be \( K' \)-quasiconformal.}

\begin{proof}
  Fix the fullness lower bound \( \Theta = \Theta_{d} \) as in Lemma \ref{thm: lem fine quasielliptic triangulations}, so that there exist arbitrarily fine internally \( \Theta \)-full triangulations.
  Fix any \( \alpha \in (0,1) \).
  By Lemma \ref{thm: lem small dilatation preserves fullness}, there is a conformal dilatation \( K > 1 \) such that on any domain, small enough \( \Theta \)-full simplices in a compact subset are sent by all \( K \)-quasiconformal maps to \( \alpha\Theta \)-full simplices.
  It will suffice to choose any \( K_{\pl} < K \).

  Let \( K' = K / K_{\pl} \), so that \( 1 < K' < K \) and \( K = K_{\pl} K' \).
  Given any \( K_{\pl} \)-quasiconformal homeomorphism \( f \colon M \to N \), we will demonstrate a radius \( \delta > 0 \) below which every \( \delta \)-fine, \( \Theta \)-full triangulation has secant approximation that is a \( \pl \)-homeomorphism.

  First note that, because \( M \) is compact, \( f \) is uniformly continuous, and any bound on radius in \( N \) can be enforced by a bound on radius in \( M \);
  we will use this without further comment.
  By Lemma \ref{thm: lem small exponentials are near isometric}, we may choose \( \delta' \) small enough that for every \( x \in M \), the exponential map \( T_{x} M \to M \) is \( \sqrt{K'} \)-quasiconformal on the \( \delta' \)-ball about origin.
  We may similarly choose \( \delta' \) small enough that the image of every \( B(x,\delta') \) in \( M \) is sent by \( f \) to small enough domain that the exponential map at \( f(x) \) is \( \sqrt{K'} \)-quasiconformal onto \( f(B(x,\delta')) \).
  In doing so, it holds that for each \( x \in M \), the restriction to \( B(x, \delta') \) of
  \[ f_{x} := \exp^{-1}_{f(x)} \composed f \composed \exp_{x} \]
  is \( K \)-quasiconformal.

  Choose a finite subcover of \( M \) by balls \( U_{i} = B(x_{i}, \delta') \) such that the interiors of the closed sub-balls \( C_{i} = \bar{B}(x_{i}, \delta' / 2) \) cover \( M \).
  Because of the choice of \( K \) (Lemma \ref{thm: lem small dilatation preserves fullness}), we may find a smaller radius \( \delta \) so that each map \( f_{x_{i}} \) sends every \( \delta \)-small and \( \Theta \)-full simplex \( \sigma \) in \( \exp^{-1}_{x_{i}} (C_{i}) \) to a simplex \( f_{x_{i}}(\sigma) \) of fullness at least \( \alpha \Theta \) in \( \exp^{-1}_{f(x_{i})}(f(C_{i})) \).
  By Corollary \ref{thm: cor internal ambient fullness commensurate}, we may choose \( \delta \) small enough that all such \( f(\sigma) \) are of ambient fullness at least \( \alpha^{2} \Theta \).
  Furthermore, by Lemma \ref{thm: lemma small normal projections}, there is \( \delta \) small enough that the normal projection of \( f(\sigma) \) is a positively oriented embedding for any \( \delta \)-small simplex in \( M \).
  Finally, choose \( \delta \) smaller than the Lebesgue number of the covering by \( C_{i} \), so that any \( \delta \)-small simplex is fully contained in at least one \( C_{i} \).

  This \( \delta \) chosen, for any \( \delta \)-fine and \( \Theta \)-full triangulation \( T \) of \( M \), the secant approximation of \( f \) along \( T \) will be a \( \pl \) mapping whose restriction to each simplex is an embedding of positive orientation.
  For possibly smaller \( \delta \), the \( \pl \)-approximation will have topological degree \( 1 \), and thus be a \( \pl \)-homeomorphism (Lemma \ref{thm-lem locally degree 1 map is homeo}).
  Finally, for any given \( \epsilon > 0 \), yet smaller \( \delta \) will suffice for the \( \pl \)-homeomorphism to be \( \epsilon \)-close to \( f \).

  The claim that the dilatation of the \( \pl \)-approximation can be made small follows from Lemma \ref{thm: lem control on affine distortion}.
  For \( f \) of sufficiently small dilatation and on small enough simplices, the \( \pl \)-approximation can be made to barely distort the shape of each simplex, and so to have small dilatation on each simplex.
\end{proof}

\section{Homeomorphisms far from being piecewise linear}\label{sec: non plizable mappings}

We turn to the proof of Theorem \ref{thm: non pl approximable homeo}, the existence of an orientation preserving quasiconformal homeomorphism \( f \colon M \to M \) that cannot be approximated by any \( \pl \)-homeomorphism.
It is a quick corollary of several well understood results in differential topology, which we recall now.
Throughout, we work only with closed manifolds.

\subsection{Local contractibility of homeomorphisms}\label{sec: local contractibility of homeos}

For a compact manifold \( M \), the group \( \Homeo(M) \) of self-homeomorphisms of \( M \) is well known to be a topological group when given the \( C^{0} \) topology.
For this, fix any continuous metric \( d \) on \( M \) and let the opens of the \( C^{0} \) topology about a given \( f \in \Homeo(M) \) be generated by the \( \epsilon \)-neighborhoods
\[ \{ g \in \Homeo(M) \colon d(f(x), g(x)) < \epsilon \quad \mbox{ for all } \quad x \in M \} \]
for each \( \epsilon > 0 \).
We say that \( g \) in such a neighborhood is \emph{\( \epsilon \)-close to} \( f \).
It is not difficult to check that this topology on \( \Homeo(M) \) does not depend on the choice of \( d \) (because \( M \) is compact).

The following Theorem was shown independently by \v{C}ernavski\u{\i} \cite{Cernavski1969--LocalContractibilityGroupHomeomorphismsManifold} and Edwards-Kirby \cites{Edwards-Kirby1971--DeformationsSpacesImbeddings}.
\begin{theorem}\label{thm-local contractibility of homeomorphisms}
  For a closed topological manifold \( M \), the homeomorphism group \( \Homeo(M) \) is locally contractible.
\end{theorem}

This statement can be rephrased in terms of isotopies between homeomorphisms.
Given two homeomorphic manifolds \( M, V \), say that a homeomorphism
\[ H \colon M \times [0,1] \to V \times [0,1] \]
is fibered over \( [0,1] \) if it preserves each fiber over \( [0,1] \).
\begin{definition}
  Fix two homeomorphisms \( h_{0}, h_{1} \colon M \to V \).
  An \emph{isotopy} from \( h_{0} \) to \( h_{1} \) is a homeomorphism
  \[ H \colon M \times [0,1] \to V \times [0,1] \]
  fibered over \( [0,1] \) and such that \( H|_{M \times \{i\}} = h_{i} \) for \( i = 0,1 \).
\end{definition}
For each \( t \in [0,1] \) one has the homeomorphism
\[ H_{t} = H(\cdot, t) \colon M \times \{t\} \to V \times \{t\} , \]
and so one may interpret an isotopy as a continuous path from \( h_{0} \) to \( h_{1} \) through homeomorphisms \( M \to V \).

In this language, Theorem \ref{thm-local contractibility of homeomorphisms} implies that every homeomorphism \( f \in \Homeo(M) \) admits a neighborhood comprising elements isotopic to \( f \).
In particular, if a homeomorphism \( f \colon M \to M \) is approximable by \( \pl \)-homeomorphisms, then there exists a close enough \( \pl \)-homeomorphism \( f_{\pl} \) to lie in its contractible neighborhood, and such \( f_{\pl} \) is isotopic to \( f \).
We need the contrapositive: \emph{if \( f \) is not isotopic to any \( \pl \)-homeomorphism, then it cannot be approximated by \( \pl \)-homeomorphisms.}

\subsection{The non-isotopic map}
Given a topological manifold \( V \) defined by an atlas of charts, recall that a \emph{\( \pl \)-structure on \( V \)} is a sub-atlas for which all transition maps are \( \pl \)-homeomorphisms.
In this case, we say that the topological manifold \( V \) underlies the \( \pl \)-structure.
The classical question, the so called Hauptvermutung, is whether \( V \) admits such a \( \pl \) reduction of charts, and if so, whether it is unique (up to \( \pl \)-homeomorphism).
The answer to this question is understood in general, with both existence and uniqueness depending on the topology of \( V \).
For a textbook exposition of the material, see \cite{Rudyak2016--PiecewiseLinearStructuresTopologicalManifolds}.

The overarching strategy is to reduce to a finer question that can be answered using the tools of classifying spaces and obstruction theory.
Without going into the full details, we make some definitions.
\begin{definition}
  Given two \( \pl \) manifolds \( M_{0}, M_{1} \) with the same underlying topological manifold \( V \), a \emph{\( \pl \)-concordance} between \( M_{0} \) and \( M_{1} \) is a \( \pl \)-structure on \( V \times [0,1] \) and \( \pl \)-homeomorphisms \( i_{0} \colon M_{0} \to V \times \{0\} \) and \( i_{1} \colon M_{1} \to V \times \{1\} \).
\end{definition}
In the definition, it is used the fact that a \( \pl \)-structure on a manifold (the product) induces a \( \pl \)-structure on its boundary.

The following is Definition 1.4.2 of \cite{Rudyak2016--PiecewiseLinearStructuresTopologicalManifolds}.
\begin{definition}[]
  A \emph{\( \pl \)-structuralization} of a topological manifold \( V \) is a \( \pl \)-manifold \( M \) and a homeomorphism
  \[ h \colon M \to V . \]
  Two \( \pl \)-structuralizations \( h_{i} \colon M_{i} \to V \) for \( i = 0,1 \) are \emph{concordant} if there exists a \( \pl \)-concordance \( M \times [0,1] \) between \( M_{0} \) and \( M_{1} \) as well as a \( \pl \)-structuralization
  \[ H \colon M \times [0,1] \to V \times [0,1] \]
  that restricts to the \( \pl \)-structuralizations \( h_{i} \) on boundary components, \( h_{0} = H \composed i_{0} \) and \( h_{1} = H \composed i_{1} \).
\end{definition}

The distinction between a \( \pl \)-structure and \( \pl \)-structuralization is somewhat subtle; the key point here is that there is a forgetful mapping from (concordance classes of) \( \pl \)-structuralizations to \( \pl \)-structures (up to \( \pl \)-homeomorphism).
The failure of this map to be injective is what gives homeomorphisms between \( \pl \)-structures that are not isotopic to any \( \pl \)-homeomorphism.
\begin{lemma}\label{thm: lem on concordances}
  Fix two structuralizations \( h_{0}, h_{1} \colon M \to V \) (in particular, with the same \( \pl \)-structure \( M \) on domain) and let
  \[ f = h^{-1}_{1} h_{0} \colon M \to M . \]
  \begin{enumerate}
  \item
    If \( h_{0}, h_{1} \) are isotopic, then they are concordant.

  \item
    If \( f \) is a \( \pl \)-homeomorphism, then \( h_{0}, h_{1} \) are concordant.

  \item
    If \( f \) is isotopic to a \( \pl \)-homeomorphism, then \( h_{0}, h_{1} \) are concordant.
  \end{enumerate}
\end{lemma}
\begin{proof}
  (1) follows from the definition; \( M \times [0,1] \) has a \( \pl \)-structure as the product of \( \pl \)-manifolds, and the isotopy between \( h_{0}, h_{1} \) gives the concordance.
  (2) similarly follows from the definition by using the mapping cylinder of \( f \).
  (3) is a consequence of (1) and (2) and the fact that concordance is an equivalence relation.
  Precisely, if \( f \) is isotopic to a \( \pl \)-homeomorphism \( F \), then \( h_{1} = h_{0} f^{-1} \) is concordant to \( h_{0} F^{-1} \), which is concordant to \( h_{0} \).
\end{proof}

Combining (3) with the remark at the end of subsection \ref{sec: local contractibility of homeos}, we see that \emph{any pair of non-concordant structuralizations \( M \to V \) determine a homeomorphism \( M \to M \) that cannot be approximated by a \( \pl \)-homeomorphism.}

The existence of such pairs of non-concordant structuralizations follows from the classification of concordance classes of \( \pl \)-structuralizations.
\begin{theorem}[\cite{Kirby-Siebenmann1969--TriangulationManifoldsHauptvermutung}]\label{thm: classification pl structuralizations}
  Given a closed topological manifold \( V \) of dimension at least \( 5 \), the single obstruction to the existence of \( \pl \)-structuralizations on \( V \) is an element of \( H^{4}(V; \Z/2\Z) \).
  If this vanishes, then the set of \( \pl \)-structuralizations up to concordance is in bijection with the group \( H^{3}(V;\Z/2\Z) \).
\end{theorem}

This Theorem essentially answers the Hauptvermutung for manifolds.
The domain of each \( \pl \)-structuralization of \( V \) is a \( \pl \)-structure on \( V \), any two of which may or may not be \( \pl \)-homeomorphic to each other.
The set of \( \pl \)-structures on \( V \) is thus a quotient of the set of \( \pl \)-structuralizations.

Earlier work of Sullivan shows that manifolds with more than one structuralization may nonetheless have a unique \( \pl \)-structure.
\begin{theorem}[\cite{Sullivan1967--HauptvermutungManifolds}]\label{thm-Sullivan homeos homotopic to PLH}
  Let \( M, M' \) be closed \( \pl \)-manifolds of dimension \( \ge 5 \), simply connected and two-torsion free in third homology.
  Any homeomorphism \( f \colon M \to M' \) is homotopic to a \( \pl \)-homeomorphism.
  (In particular, \( M \) is \( \pl \)-homeomorphic to \( M' \).)
\end{theorem}

This is now combined with the previous Theorem to give non-\( \pl \)-approximable homeomorphisms.
In each dimension \( n \ge 5 \), let \( M = S^{3} \times S^{n-3} \), which is smooth and thus admits a \( \pl \)-structure.
From Theorem \ref{thm-Sullivan homeos homotopic to PLH} we see that this \( \pl \)-structure is unique.
On the other hand, by Theorem \ref{thm: classification pl structuralizations} there are at least 2 non-concordant \( \pl \)-structuralizations of \( M \), say \( h_{1}, h_{2} \colon M \to M \).
By Lemma \ref{thm: lem on concordances} and the remark after, the composition \( f = h_{2}^{-1} h_{1} \) is a homeomorphism of \( M \) that is not approximable by any \( \pl \)-homeomorphism.

It may occur that the so obtained \( f \) is orientation reversing, but this can be fixed as follows.
Any reflection of \( S^{3} \) across a hyperplane reverses orientation, and induces likewise an orientation reversing \( \pl \)-homeomorphism \( h \) on \( M \).
The composition \( h f \) is orientation preserving, and continues to not be \( \pl \)-homeomorphism-approximable, because a good \( C^{0} \)-approximation to \( h f \) postcomposes with \( h^{-1} \) to a good approximation of \( f \).

\subsection{The quasiconformal example}
We have so far the existence, in each dimension at least 5, of homeomorphisms \( f \colon M \to M \) that are at least \( \epsilon \)-far from all \( \pl \)-homeomorphisms, some \( \epsilon > 0 \).
It remains to show that \( f \) can be taken quasiconformal.
This follows simply from the fact that all homeomorphisms can be arbitrarily well \( C^{0} \)-approximated by quasiconformal mappings (of large dilatation).

\begin{theorem}[\cites{Sullivan1979--HyperbolicGeometryHomeomorphisms, Tukia-Vaeisaelae1981--LipschitzQuasiconformalApproximationExtension}]\label{thm-QC approximation of homeos}
  Fix \( d \)-dimensional closed manifolds \( M, N \), with \( d \neq 4 \).
  For any homeomorphism \( f \colon M \to N \) and any \( \epsilon > 0 \), there exists a quasiconformal homeomorphism \( f' \colon M \to N \) that is \( \epsilon/2 \)-close to \( f \).
\end{theorem}

\bibliography{references}

@book {Ahlfors1966--LecturesQuasiconformalMappings,
    AUTHOR = {Ahlfors, Lars V.},
     TITLE = {Lectures on Quasiconformal Mappings},
    SERIES = {Van Nostrand Mathematical Studies, No. 10},
 PUBLISHER = {D. Van Nostrand Co., Inc., Toronto, Ont.-New York-London},
      YEAR = 1966,
     PAGES = {v+146},
   MRCLASS = {30.47},
  MRNUMBER = 0200442,
MRREVIEWER = {P. Caraman},
}

@article {Cernavski1969--LocalContractibilityGroupHomeomorphismsManifold,
    AUTHOR = {{\v C}ernavski{\u\i}, A. V.},
     TITLE = {Local Contractibility of the Group of Homeomorphisms of a Manifold},
   JOURNAL = {Mat. Sb. (N.S.)},
  FJOURNAL = {Matematicheski{\u\i} Sbornik. Novaya Seriya},
    VOLUME = {79(121)},
      YEAR = 1969,
     PAGES = {307--356},
      ISSN = {0368-8666},
   MRCLASS = {57.05 (54.00)},
  MRNUMBER = 259925,
MRREVIEWER = {D.\ G.\ Ebin},
}

@article {Edwards-Kirby1971--DeformationsSpacesImbeddings,
    AUTHOR = {Edwards, Robert D. and Kirby, Robion C.},
     TITLE = {Deformations of Spaces of Imbeddings},
   JOURNAL = {Ann. of Math. (2)},
  FJOURNAL = {Annals of Mathematics. Second Series},
    VOLUME = 93,
      YEAR = 1971,
     PAGES = {63--88},
      ISSN = {0003-486X},
   MRCLASS = {57.01},
  MRNUMBER = 283802,
MRREVIEWER = {K.\ Lamotke},
       DOI = {10.2307/1970753},
       URL = {https://doi.org/10.2307/1970753},
}

@article {Fletcher-Markovic2012--Decomposing,
    AUTHOR = {Fletcher, Alastair and Markovic, Vladimir},
     TITLE = {Decomposing Diffeomorphisms of the Sphere},
   JOURNAL = {Bull. Lond. Math. Soc.},
  FJOURNAL = {Bulletin of the London Mathematical Society},
    VOLUME = 44,
      YEAR = 2012,
    NUMBER = 3,
     PAGES = {599--609},
      ISSN = {0024-6093,1469-2120},
   MRCLASS = {30C65 (58C99)},
  MRNUMBER = 2967005,
MRREVIEWER = {Kevin\ Wildrick},
       DOI = {10.1112/blms/bdr111},
       URL = {https://doi.org/10.1112/blms/bdr111},
}

@book {Gehring-Martin-Palka2017--IntroductionTheoryHigherDimensionalQuasiconformalMappings,
    AUTHOR = {Gehring, Frederick W. and Martin, Gaven J. and Palka, Bruce
              P.},
     TITLE = {An Introduction to the Theory of Higher-Dimensional Quasiconformal Mappings},
    SERIES = {Mathematical Surveys and Monographs},
    VOLUME = 216,
 PUBLISHER = {American Mathematical Society, Providence, RI},
      YEAR = 2017,
     PAGES = {ix+430},
      ISBN = {978-0-8218-4360-4},
   MRCLASS = {30-02 (30C62 30C65)},
  MRNUMBER = 3642872,
MRREVIEWER = {David Matthew Freeman},
       DOI = {10.1090/surv/216},
       URL = {https://doi.org/10.1090/surv/216},
}

@article {Gehring1962--RingsQuasiconformalMappingsSpace,
    AUTHOR = {Gehring, F. W.},
     TITLE = {Rings and Quasiconformal Mappings in Space},
   JOURNAL = {Trans. Amer. Math. Soc.},
  FJOURNAL = {Transactions of the American Mathematical Society},
    VOLUME = 103,
      YEAR = 1962,
     PAGES = {353--393},
      ISSN = {0002-9947},
   MRCLASS = {30.47},
  MRNUMBER = 139735,
MRREVIEWER = {L. V. Ahlfors},
       DOI = {10.2307/1993834},
       URL = {https://doi.org/10.2307/1993834},
}

@inproceedings {Gehring1987--TopicsQuasiconformalMappings,
    AUTHOR = {Gehring, F. W.},
     TITLE = {Topics in Quasiconformal Mappings},
 BOOKTITLE = {Proceedings of the {I}nternational {C}ongress of
              {M}athematicians, {V}ol. 1, 2 ({B}erkeley, {C}alif., 1986)},
     PAGES = {62--80},
 PUBLISHER = {Amer. Math. Soc., Providence, RI},
      YEAR = 1987,
   MRCLASS = {30C60 (30-02)},
  MRNUMBER = 934216,
MRREVIEWER = {Matti Vuorinen},
}

@incollection {Gehring2005--QuasiconformalMappingsEuclideanSpaces,
    AUTHOR = {Gehring, F. W.},
     TITLE = {Quasiconformal Mappings in {E}uclidean Spaces},
 BOOKTITLE = {Handbook of complex analysis: geometric function theory.
              {V}ol. 2},
     PAGES = {1--29},
 PUBLISHER = {Elsevier Sci. B. V., Amsterdam},
      YEAR = 2005,
   MRCLASS = {30C65 (30F45 35J60)},
  MRNUMBER = 2121856,
MRREVIEWER = {Leonid V. Kovalev},
       DOI = {10.1016/S1874-5709(05)80005-8},
       URL = {https://doi.org/10.1016/S1874-5709(05)80005-8},
}

@article{He-Liu2019--FactoringHigherDimensionalQuasiconformalMappings,
    AUTHOR = {He, Zhengxu and Liu, Jinsong},
     TITLE = {Factoring the Higher Dimensional Quasiconformal Mappings},
   JOURNAL = {Trans. Amer. Math. Soc.},
  FJOURNAL = {Transactions of the American Mathematical Society},
    VOLUME = 372,
      YEAR = 2019,
    NUMBER = 8,
     PAGES = {5341--5353},
      ISSN = {0002-9947},
   MRCLASS = {30C65 (32A30)},
  MRNUMBER = 4014278,
MRREVIEWER = {Anca Andrei},
       DOI = {10.1090/tran/7679},
       URL = {https://doi.org/10.1090/tran/7679},
}

@book {Hudson1969--PiecewiseLinearTopology,
    AUTHOR = {Hudson, J. F. P.},
     TITLE = {Piecewise Linear Topology},
 PUBLISHER = {W. A. Benjamin, Inc., New York-Amsterdam},
      YEAR = 1969,
     PAGES = {ix+282},
   MRCLASS = {57.01},
  MRNUMBER = 248844,
MRREVIEWER = {R.\ B.\ Sher},
}

@article {Kirby-Siebenmann1969--TriangulationManifoldsHauptvermutung,
    AUTHOR = {Kirby, R. C. and Siebenmann, L. C.},
     TITLE = {On the Triangulation of Manifolds and the {H}auptvermutung},
   JOURNAL = {Bull. Amer. Math. Soc.},
  FJOURNAL = {Bulletin of the American Mathematical Society},
    VOLUME = 75,
      YEAR = 1969,
     PAGES = {742--749},
      ISSN = {0002-9904},
   MRCLASS = {57.01},
  MRNUMBER = 242166,
MRREVIEWER = {R. K. Lashof},
       DOI = {10.1090/S0002-9904-1969-12271-8},
       URL = {https://doi.org/10.1090/S0002-9904-1969-12271-8},
}

@incollection {Martin2014--TheoryQuasiconformalMappingsHigherDimensionsI,
    AUTHOR = {Martin, Gaven J.},
     TITLE = {The Theory of Quasiconformal Mappings in Higher Dimensions, {I}},
 BOOKTITLE = {Handbook of {T}eichm\"{u}ller theory. {V}ol. {IV}},
    SERIES = {IRMA Lect. Math. Theor. Phys.},
    VOLUME = 19,
     PAGES = {619--677},
 PUBLISHER = {Eur. Math. Soc., Z\"{u}rich},
      YEAR = 2014,
   MRCLASS = {30C65 (35J62 35J92 37F50)},
  MRNUMBER = 3289712,
MRREVIEWER = {Jarmo J\"{a}\"{a}skel\"{a}inen},
       DOI = {10.4171/117-1/15},
       URL = {https://doi.org/10.4171/117-1/15},
}

@article {Moise1952--AffineStructuresManifoldsIV,
    AUTHOR = {Moise, Edwin E.},
     TITLE = {Affine Structures in {$3$}-Manifolds. {IV}. {P}iecewise Linear Approximations of Homeomorphisms},
   JOURNAL = {Ann. of Math. (2)},
  FJOURNAL = {Annals of Mathematics. Second Series},
    VOLUME = 55,
      YEAR = 1952,
     PAGES = {215--222},
      ISSN = {0003-486X},
   MRCLASS = {56.0X},
  MRNUMBER = 46644,
MRREVIEWER = {S.\ S.\ Cairns},
       DOI = {10.2307/1969775},
       URL = {https://doi.org/10.2307/1969775},
}

@article {Nash1956--ImbeddingProblemRiemannianManifolds,
    AUTHOR = {Nash, John},
     TITLE = {The Imbedding Problem for {R}iemannian Manifolds},
   JOURNAL = {Ann. of Math. (2)},
  FJOURNAL = {Annals of Mathematics. Second Series},
    VOLUME = 63,
      YEAR = 1956,
     PAGES = {20--63},
      ISSN = {0003-486X},
   MRCLASS = {53.1X},
  MRNUMBER = 75639,
MRREVIEWER = {J.\ Schwartz},
       DOI = {10.2307/1969989},
       URL = {https://doi.org/10.2307/1969989},
}

@book {Rourke-Sanderson1972--IntroductionPiecewiseLinearTopology,
    AUTHOR = {Rourke, C. P. and Sanderson, B. J.},
     TITLE = {Introduction to Piecewise-Linear Topology},
    SERIES = {Ergebnisse der Mathematik und ihrer Grenzgebiete [Results in
              Mathematics and Related Areas]},
    VOLUME = {Band 69},
 PUBLISHER = {Springer-Verlag, New York-Heidelberg},
      YEAR = 1972,
     PAGES = {viii+123},
   MRCLASS = {57CXX},
  MRNUMBER = 350744,
MRREVIEWER = {R.\ C.\ Kirby},
}

@book {Rudyak2016--PiecewiseLinearStructuresTopologicalManifolds,
    AUTHOR = {Rudyak, Yuli},
     TITLE = {Piecewise Linear Structures on Topological Manifolds},
 PUBLISHER = {World Scientific Publishing Co. Pte. Ltd., Hackensack, NJ},
      YEAR = 2016,
     PAGES = {xxii+106},
      ISBN = {978-981-4733-78-6},
   MRCLASS = {57Q25 (55R25 57Q15 57Qxx)},
  MRNUMBER = 3467983,
MRREVIEWER = {Laurence\ R.\ Taylor},
       DOI = {10.1142/9887},
       URL = {https://doi.org/10.1142/9887},
}

@article {Shikata1966--DistanceFunctionSetDifferentiableStructures,
    AUTHOR = {Shikata, Yoshihiro},
     TITLE = {On a Distance Function on the Set of Differentiable Structures},
   JOURNAL = {Osaka Math. J.},
  FJOURNAL = {Osaka Mathematical Journal},
    VOLUME = 3,
      YEAR = 1966,
     PAGES = {65--79},
      ISSN = {0388-0699},
   MRCLASS = {57.10},
  MRNUMBER = 202149,
MRREVIEWER = {R.\ H.\ Szczarba},
       URL = {http://projecteuclid.org/euclid.ojm/1200691577},
}

@article {Sullivan1967--HauptvermutungManifolds,
    AUTHOR = {Sullivan, Dennis},
     TITLE = {On the {H}auptvermutung for Manifolds},
   JOURNAL = {Bull. Amer. Math. Soc.},
  FJOURNAL = {Bulletin of the American Mathematical Society},
    VOLUME = 73,
      YEAR = 1967,
     PAGES = {598--600},
      ISSN = {0002-9904},
   MRCLASS = {57.05},
  MRNUMBER = 212811,
MRREVIEWER = {Edward\ M.\ Brown},
       DOI = {10.1090/S0002-9904-1967-11764-6},
       URL = {https://doi.org/10.1090/S0002-9904-1967-11764-6},
}

@inproceedings {Sullivan1979--HyperbolicGeometryHomeomorphisms,
    AUTHOR = {Sullivan, Dennis},
     TITLE = {Hyperbolic Geometry and Homeomorphisms},
 BOOKTITLE = {Geometric topology ({P}roc. {G}eorgia {T}opology {C}onf.,
              {A}thens, {G}a., 1977)},
     PAGES = {543--555},
 PUBLISHER = {Academic Press, New York-London},
      YEAR = 1979,
   MRCLASS = {57N99 (30C60 57Q25 57R15 58A12 58G10)},
  MRNUMBER = 537749,
MRREVIEWER = {Christopher W. Stark},
}

@article {Tukia-Vaeisaelae1981--LipschitzQuasiconformalApproximationExtension,
    AUTHOR = {Tukia, P. and V\"ais\"al\"a, J.},
     TITLE = {Lipschitz and Quasiconformal Approximation and Extension},
   JOURNAL = {Ann. Acad. Sci. Fenn. Ser. A I Math.},
  FJOURNAL = {Annales Academiae Scientiarum Fennicae. Series A I.
              Mathematica},
    VOLUME = 6,
      YEAR = 1981,
    NUMBER = 2,
     PAGES = {303--342},
      ISSN = {0066-1953},
   MRCLASS = {57Q25 (30C60 54E40 57Q30)},
  MRNUMBER = 658932,
MRREVIEWER = {David\ C.\ Wilson},
       DOI = {10.5186/aasfm.1981.0626},
       URL = {https://doi.org/10.5186/aasfm.1981.0626},
}

@book {Vaisala1971--LecturesDimensionalQuasiconformalMappings,
    AUTHOR = {V\"{a}is\"{a}l\"{a}, Jussi},
     TITLE = {Lectures on {$n$}-Dimensional Quasiconformal Mappings},
    SERIES = {Lecture Notes in Mathematics, Vol. 229},
 PUBLISHER = {Springer-Verlag, Berlin-New York},
      YEAR = 1971,
     PAGES = {xiv+144},
   MRCLASS = {30A60},
  MRNUMBER = 0454009,
MRREVIEWER = {F. W. Gehring},
}

@article {Whitehead1940--C1Complexes,
    AUTHOR = {Whitehead, J. H. C.},
     TITLE = {On {$C^1$}-Complexes},
   JOURNAL = {Ann. of Math. (2)},
  FJOURNAL = {Annals of Mathematics. Second Series},
    VOLUME = 41,
      YEAR = 1940,
     PAGES = {809--824},
      ISSN = {0003-486X},
   MRCLASS = {56.0X},
  MRNUMBER = 2545,
MRREVIEWER = {R.\ H.\ Fox},
       DOI = {10.2307/1968861},
       URL = {https://doi.org/10.2307/1968861},
}

@article {Whitney1936--DifferentiableManifolds,
    AUTHOR = {Whitney, Hassler},
     TITLE = {Differentiable Manifolds},
   JOURNAL = {Ann. of Math. (2)},
  FJOURNAL = {Annals of Mathematics. Second Series},
    VOLUME = 37,
      YEAR = 1936,
    NUMBER = 3,
     PAGES = {645--680},
      ISSN = {0003-486X,1939-8980},
   MRCLASS = {99-04},
  MRNUMBER = 1503303,
       DOI = {10.2307/1968482},
       URL = {https://doi.org/10.2307/1968482},
}

@book {Whitney1957--GeometricIntegrationTheory,
    AUTHOR = {Whitney, Hassler},
     TITLE = {Geometric Integration Theory},
 PUBLISHER = {Princeton University Press, Princeton, NJ},
      YEAR = 1957,
     PAGES = {xv+387},
   MRCLASS = {53.0X},
  MRNUMBER = 87148,
MRREVIEWER = {H.\ Samelson},
}
\end{document}